\documentclass[12pt,reqno]{amsart}

\usepackage[T1]{fontenc}
\usepackage[utf8]{inputenc}
\usepackage{lmodern}
\usepackage{amsmath,amssymb,mathtools,mathrsfs}
\usepackage{enumitem}
\usepackage{microtype}
\usepackage[colorlinks=true,linkcolor=blue,citecolor=blue,urlcolor=blue]{hyperref}

\allowdisplaybreaks
\newtheorem{theorem}{Theorem}[section]
\newtheorem{theoremintro}{Theorem}

\newtheorem{proposition}[theorem]{Proposition}
\newtheorem{lemma}[theorem]{Lemma}
\newtheorem{corollary}[theorem]{Corollary}
\theoremstyle{definition}
\newtheorem{definition}[theorem]{Definition}
\theoremstyle{remark}

\newcommand{\NN}{\mathbb N}
\newcommand{\Fp}{\mathcal F_p}
\newcommand{\Lip}{\operatorname{Lip}}
\newcommand{\Id}{\operatorname{Id}}
\newcommand{\dist}{\operatorname{dist}}
\newcommand{\spn}{\operatorname{span}}
\newcommand{\olsp}{\overline{\operatorname{span}}}

\newcommand{\norm}[1]{\left\lVert #1\right\rVert}
\newcommand{\abs}[1]{\left\lvert #1\right\rvert}
\newcommand{\set}[1]{\left\{#1\right\}}

\newcommand{\vertiii}[1]{\left\lvert\!\left\lvert\!\left\lvert #1\right\rvert\!\right\rvert\!\right\rvert}

\title[The Schur $p$-property of Lipschitz-free $p$-spaces]
{Structural consequences of the Schur $p$-property for Lipschitz-free $p$-spaces}

\author[F. Albiac]{Fernando Albiac}
\address{Institute for Advanced Materials and Mathematics (INAMAT$^{2}$) and Department of Mathematics, Statistics and Computer Sciences\\ Public University of Navarre\\
Campus de Arrosad\'{\i}a, 31006 Pamplona\\ Spain}
\email{fernando.albiac@unavarra.es}
\author[J. L. Ansorena]{Jos\'e L. Ansorena}
\address{Department of Mathematics and Computer Sciences\\
Universidad de La Rioja\\
Logro\~no\\
26004 Spain}
\email{joseluis.ansorena@unirioja.es}

\author[M. C\'uth]{Marek C\'uth}
\address{Faculty of Mathematics and Physics, Department of Mathematical Analysis\\
Charles University\\
186 75 Praha 8\\
Czech Republic}
\email{cuth@karlin.mff.cuni.cz}

\date{Working draft, September 2026}


\begin{document}

\begin{abstract}
Let $0<p<q\leq 1$. We show that the Schur $p$-property provides a
powerful structural principle for Lipschitz-free $p$-spaces. Our main
result asserts that $\mathcal{F}_p(M)$ has the Schur $p$-property for
every $q$-metric space $M$; when $M$ is compact, it has the strong
Schur $p$-property, with a constant depending only on $p$ and $q$.
As consequences, $\mathcal{F}_p(M)$ is $\ell_p$-saturated, contains no
isomorphic copy of an infinite-dimensional $r$-Banach space for
$p<r\leq 1$, and every bounded operator from a $p$-Banach space into
$\mathcal{F}_p(M)$ is either compact or fixes a copy of $\ell_p$.
These results settle Questions~6.1, 6.2, 6.5, and 6.6 from our recent
work \cite{AABCSchur} on the Schur $p$-property, together with several related problems.
They also confirm a prediction of  Kalton and the first-named author from 2009: no
infinite-dimensional $q$-Banach space has the $p$-Lipschitz lifting
property. Further applications reveal a sharp contrast between the
linear and Lipschitz structures of nonlocally convex spaces.
\end{abstract}

\thanks{F.\@ Albiac acknowledges the support of the Spanish Ministry for Science, Innovation, and Universities under Grant PID2025-167660NB-I00 funded by MICIU/AEI/10.13039/501100011033 and ERDF/EU. F.\@ Albiac and J.\@ L.\@ Ansorena acknowledge the support of the Spanish Ministry for Science, Innovation, and Universities under Grant PGC2018-095366-B-I00 for \emph{Functional Analysis Techniques in Approximation Theory and Applications (TAFPAA)}}

\maketitle

\section{Introduction}

\noindent Throughout the paper, all vector spaces are real and $0<p<1$ is fixed. The
Schur $p$-property was introduced in \cite[Definition~2.2]{AABCSchur}: a
$p$-Banach space has this property if every infinite bounded uniformly
separated subset contains a sequence equivalent to the canonical basis of
$\ell_p$.

The purpose of this paper is to establish the following universal result.
A $q$-metric is a distance $d$ such that $d^q$ is a metric; in particular,
the case $q=1$ is the usual metric case.

\begin{theoremintro}\label{thm:main-intro}
Let $0<p<q\leq1$ and let $M$ be a pointed $q$-metric space. Then
$\Fp(M)$ has the Schur $p$-property.
\end{theoremintro}

There are no boundedness or separability assumptions on $M$. Passing to
the completion does not change its Lipschitz free $p$-space, as explained in
Section~\ref{sec:preliminaries}. For compact bases we obtain a
quantitative strengthening.

\begin{theoremintro}\label{thm:compact-intro}
For every $0<p<q\leq1$ there is a constant $K_{p,q}$ such that
$\Fp(K)$ has the $K_{p,q}$-strong Schur $p$-property whenever $K$ is a
compact pointed $q$-metric space. In particular, $\Fp(K)$ has
the strong Schur $p$-property for every compact metric space $K$.
\end{theoremintro}

The constant is independent of the dimension and covering numbers of the
base. More explicitly, if $\Gamma_p$ is the uniform canonical-embedding
constant from Theorem~\ref{thm:canonical-input}, we may take
\[
 K_{p,q}=\Gamma_p
 \left(\frac1{1-16^{p-q}}+1+\frac1{1-16^{-p}}\right)^{1/p}.
\]
As a consequence, the strong Schur $p$-property also holds for free
$p$-spaces over arbitrary subsets of finite-dimensional normed spaces.

Among the principal structural consequences are the following.

\begin{theoremintro}\label{thm:consequences-intro}
Let $0<p<q\leq1$ and let $M$ be a $q$-metric space. Then $\Fp(M)$ is $\ell_p$-saturated and contains no isomorphic copy of an infinite-dimensional $r$-Banach space for any $p<r\leq1$. Moreover, every bounded operator from a $p$-Banach space into $\Fp(M)$ is either compact or fixes a copy of $\ell_p$.
\end{theoremintro}

Our results lead to answers to \cite[Questions 6.1, 6.2, 6.5, 6.6]{AABCSchur}. Several more applications are gathered in Section~\ref{sec:applications}.

The paper is organized as follows. Section~\ref{sec:preliminaries}
collects the preliminary results. Section~\ref{sec:detectors} develops
the diagonal detector and lacunary assembly lemmas.
Section~\ref{sec:quotients} constructs the quotient detectors and proves
the strong Schur $p$-property for compact bases.
Section~\ref{sec:bounded} treats bounded bases, and
Section~\ref{sec:general} proves the theorem for arbitrary $q$-metric
spaces. Section~\ref{sec:applications} develops the applications.

\section{Preliminaries}
\label{sec:preliminaries}

A \emph{$p$-Banach space} is a complete quasi-normed space $X$ whose
quasi-norm satisfies
\[
 \norm{x+y}^p\leq \norm{x}^p+\norm{y}^p
 \qquad(x,y\in X).
\]
A pointed $p$-metric space is a triple $(M,d,0)$ such that $d$ is symmetric,
vanishes exactly on the diagonal, and $d^p$ is a metric. Every $q$-metric space is a $p$-metric space when $0<p\leq q\leq1$,
and the same implication holds for $q$-Banach and $p$-Banach spaces.
In particular, every metric space is a $p$-metric space. Topological
notions for a $p$-metric $d$ refer to the topology of the metric $d^p$. For a subset $A\subseteq M$ and $r>0$, write
\[
 [A]_r=\{x\in M:\dist_d(x,A)\leq r\}.
\]
For a quasi-normed space $X$, $B_X$ and $S_X$ denote its closed unit
ball and its unit sphere, respectively.

We use the standard construction and universal property from
\cite[Theorem~4.5]{AACD}. For every pointed $p$-metric space $(M,d,0)$ there
are a $p$-Banach space $\Fp(M,d)$ and an isometric pointed map
\[
 \delta_{M,d}\colon M\longrightarrow \Fp(M,d)
\]
whose linear span is dense and such that every pointed Lipschitz map
$f\colon M\to Y$ into a $p$-Banach space has a unique bounded
linearization
\[
 \widehat f\colon\Fp(M,d)\longrightarrow Y,
 \qquad \widehat f\delta_{M,d}=f,
 \qquad \norm{\widehat f}=\Lip(f).
\]
We omit $d$ from the notation when no confusion is possible. Linearizations
respect composition.

By \cite[Proposition~4.17]{AACD}, if $A$ is a dense pointed subset of
a $p$-metric space $M$, the canonical map $\Fp(A)\to\Fp(M)$ is an
onto linear isometry. In particular, if $0<p\leq q\leq1$ and
$(\overline M,\overline d,0)$ is the completion of a pointed
$q$-metric space $(M,d,0)$, then
\begin{equation}\label{eq:completion-isometry}
 \Fp(M,d)\equiv\Fp(\overline M,\overline d)
\end{equation}
canonically and isometrically. Here the completion is obtained by
completing the metric $d^q$ and taking its $q$th root. Henceforth,
all ambient $q$-metric spaces, $0<q\leq1$, are assumed complete.
This causes no loss of generality for the free-space assertions.
Subsets carry the induced distance and need not be closed; the same
dense-subset result identifies $\Fp(A)$ isometrically with
$\Fp(\overline A)$, where $\overline A$ is its closure in a complete
ambient space.

If $A\subseteq M$ contains the base point, set
\[
 \Fp(A;M)=\olsp\set{\delta_M(a):a\in A}\subseteq\Fp(M).
\]
The inclusion $A\hookrightarrow M$ linearizes to a contraction from
$\Fp(A)$ onto $\Fp(A;M)$. We shall use the following uniform theorem.

\begin{theorem}[Uniform canonical embeddings]\label{thm:canonical-input}
For every $0<p\leq1$ there is a constant $\Gamma_p\geq1$ such that, for
every inclusion $0\in A\subseteq M$ of pointed $p$-metric spaces, the
canonical map
\[
 \iota_{A,M}\colon\Fp(A)\longrightarrow \Fp(A;M)
\]
is an isomorphism onto $\Fp(A;M)$ and
\[
 \norm{\iota_{A,M}}\leq1,
 \qquad
 \norm{\iota_{A,M}^{-1}}\leq\Gamma_p.
\]
\end{theorem}

\begin{proof}
This is precisely \cite[Theorem~3.5]{AlbiacAnsorenaExtensions}. That theorem
is formulated for arbitrary inclusions of pointed $p$-metric spaces, which
is essential below because the auxiliary distances $\rho_b$ need not be
metrics.
\end{proof}

A vector in $\Fp(M)$ is \emph{finitely supported} if it is a finite
linear combination of Dirac vectors. A finite pointed set containing all
points used in such a representation will be called a finite pointed
support of the vector. These vectors form a dense subspace of $\Fp(M)$.
Finite families of distinct nonzero Dirac vectors are linearly independent.
Indeed, on a finite pointed set $F$, every scalar function vanishing at
$0$ is Lipschitz. The coordinate functions, linearized by the universal
property, show that the vectors $\delta_F(x)$, $x\in F\setminus\{0\}$,
are independent. Theorem~\ref{thm:canonical-input} transfers this fact
to their images in any ambient free $p$-space.

\begin{definition}[\cite{AABCSchur}]
A $p$-Banach space $X$ has the \emph{Schur $p$-property} if every infinite
bounded uniformly separated subset of $X$ contains a sequence equivalent to
the canonical basis of $\ell_p$.
\end{definition}

Thus a sequence $(x_n)$ is equivalent to the canonical basis of $\ell_p$ if
there are constants $c,C>0$ such that
\begin{equation}\label{eq:ellp-equivalence}
 c^p\sum_n\abs{a_n}^p
 \leq
 \norm{\sum_n a_nx_n}^p
 \leq
 C^p\sum_n\abs{a_n}^p
\end{equation}
for every finitely supported scalar sequence $(a_n)$.

\begin{lemma}[Stability]\label{lem:Schur-stability}
The Schur $p$-property passes to closed subspaces and is invariant under
linear isomorphisms.
\end{lemma}

\begin{proof}
A bounded uniformly separated subset of a closed subspace has the same
properties in the ambient space, and the estimates
\eqref{eq:ellp-equivalence} are unchanged. If $T\colon X\to Y$ is an
isomorphism, then $T$ and $T^{-1}$ preserve boundedness and uniform
separation up to their operator norms; applying $T$ to the selected sequence
only changes the two constants in \eqref{eq:ellp-equivalence}.
\end{proof}

Finite-dimensional quasi-Banach spaces trivially have the Schur
$p$-property, because their bounded sets are totally bounded.

\section{Detector principles}
\label{sec:detectors}

The arguments below use bounded operators with values in an $\ell_p$-sum.
A positive diagonal already detects an $\ell_p$-subsequence: decay of
the coordinates of each fixed vector and a Ramsey diagonalization
supply the required small off-diagonal terms. We first prove this
principle and then record the lacunary assembly lemma used for the
quotient detectors.

\subsection{The diagonal detector lemma}

For later quantitative statements, fix $K\geq1$ and recall that a $p$-Banach space $X$ has
the \emph{$K$-strong Schur $p$-property} if, for every $\delta>0$, every
$K_0>K$, and every infinite $\delta$-separated set $A\subseteq S_X$, there
is an infinite subset $B\subseteq A$ such that
\begin{equation}\label{eq:strong-Schur-lower}
 \norm{\sum_{x\in B}a_x x}
 \geq
 \frac{\delta}{K_0 2^{1/p}}
 \left(\sum_{x\in B}\abs{a_x}^p\right)^{1/p}
\end{equation}
for every finitely supported scalar family $(a_x)_{x\in B}$. This is the
form of \cite[Definition~2.5]{AABCSchur} appropriate for $p$-Banach spaces and $K\geq1$:
the corresponding upper estimate with constant $K_0$ is automatic on the
unit sphere from the $p$-triangle inequality, since $K_0> K\geq1$. The strong Schur $p$-property passes to subspaces
and is invariant under linear isomorphisms; see the discussion following
\cite[Definition~2.5]{AABCSchur}.

\begin{lemma}[Diagonal detector lemma]
\label{lem:diagonal-detector}
\label{lem:triangular-detector}
Let $X$ be a $p$-Banach space, let $(E_n)_{n\in\NN}$ be $p$-Banach
spaces, and let
\[
 T\colon X\longrightarrow
 E:=\left(\bigoplus_{n=1}^{\infty}E_n\right)_p
\]
be bounded. Denote the coordinate projections by $\pi_n\colon E\to E_n$.
If $(x_n)$ is bounded in $X$ and
\[
 \alpha:=\limsup_{n\to\infty}\norm{\pi_nTx_n}>0,
\]
then, for every $0<c<\alpha/\norm T$, there is a subsequence
$(x_{n_j})$ such that
\begin{equation}\label{eq:triangular-lower}
 c^p\sum_j\abs{a_j}^p
 \leq\norm{\sum_j a_jx_{n_j}}^p
 \leq\left(\sup_n\norm{x_n}\right)^p\sum_j\abs{a_j}^p
\end{equation}
for every finitely supported scalar sequence $(a_j)$.
\end{lemma}

\begin{proof}
Fix $c$ as above and choose $\alpha_0$ with
$c\norm T<\alpha_0<\alpha$. The set
$I_0=\{n:\norm{\pi_nTx_n}>\alpha_0\}$ is infinite. Choose positive
numbers $(\theta_j)$ such that
\begin{equation}\label{eq:diagonal-error-budget}
 \sum_{j=1}^{\infty}\theta_j^p
 <\alpha_0^p-c^p\norm T^p.
\end{equation}

We first perform a Ramsey diagonalization. Given an infinite set
$I\subseteq\NN$ and $t>0$, color a pair $\{n,k\}\subseteq I$, $n<k$,
by whether $\norm{\pi_nTx_k}<t$ or $\norm{\pi_nTx_k}\geq t$.
Infinite Ramsey's theorem gives an infinite homogeneous subset.
The second color is impossible: if $n_1<\cdots<n_m$ belong to such
a subset, then
\[
 \norm{Tx_{n_m}}^p
 \geq\sum_{j=1}^{m-1}\norm{\pi_{n_j}Tx_{n_m}}^p
 \geq(m-1)t^p,
\]
contrary to the boundedness of $(Tx_n)$. Thus every infinite $I$
contains an infinite $J$ such that $\norm{\pi_nTx_k}<t$ whenever
$n,k\in J$ and $n<k$.

We recursively choose $n_j$ and infinite sets $I_j$. Set $n_0=0$. We recursively choose increasing indices $n_j$ and
infinite sets $I_j$. At step $j$, choose an integer $N_j>n_{j-1}$
such that
\[
 \norm{\pi_nTx_{n_i}}<\theta_j
 \qquad(n\geq N_j,\ 1\leq i<j).
\]
Such an integer exists because, for each previously chosen index
$n_i$, the inclusion $Tx_{n_i}\in E$ implies
$\norm{\pi_nTx_{n_i}}\to0$ as $n\to\infty$, and there are only
finitely many indices $i<j$. For $j=1$, this condition is empty,
and we take $N_1=1$. The set $I'_{j-1}=\{n\in I_{j-1}:n\geq N_j\}$ is infinite, since only finitely many indices have been removed
from $I_{j-1}$. Apply the preceding observation to $I'_{j-1}$
with $t=\theta_j$, obtaining an infinite set $J_j\subseteq I'_{j-1}$.
Put $n_j=\min J_j$ and $I_j=J_j\setminus\{n_j\}$. Then
\[
 \norm{\pi_{n_j}Tx_{n_j}}>\alpha_0,
 \qquad
 \norm{\pi_{n_j}Tx_{n_k}}<\theta_j\quad(k\ne j).
\]
For $k<j$ the second inequality holds by the initial discard at
step $j$; for $k>j$ it follows from $n_k\in I_j\subseteq J_j$.

For a finitely supported scalar sequence $(a_j)$, the reverse
$p$-triangle inequality in each selected coordinate gives
\begin{align*}
 \norm{T\sum_k a_kx_{n_k}}^p&\geq\sum_j\norm{\pi_{n_j}T\sum_k a_kx_{n_k}}^p\\
 &\geq\sum_j\left(
      \abs{a_j}^p\norm{\pi_{n_j}Tx_{n_j}}^p
      -\sum_{k\ne j}\abs{a_k}^p
         \norm{\pi_{n_j}Tx_{n_k}}^p
    \right)\\
 &\geq\alpha_0^p\sum_j\abs{a_j}^p
       -\sum_j\theta_j^p\sum_{k\ne j}\abs{a_k}^p\\
 &\geq\left(\alpha_0^p-\sum_j\theta_j^p\right)
       \sum_k\abs{a_k}^p
 \geq c^p\norm T^p\sum_k\abs{a_k}^p.
\end{align*}
Boundedness of $T$ gives the lower estimate
in \eqref{eq:triangular-lower}. The upper estimate follows from
the boundedness of $(x_n)$ and the $p$-triangle inequality.
\end{proof}

\subsection{Lacunary assembly}

\begin{lemma}[Lacunary assembly]\label{lem:lacunary-assembly}
Let $0<p<q\leq1$, let $M$ be a pointed $q$-metric space, let $(E_j)_{j\geq1}$ be $p$-Banach
spaces, and let $\Phi_j\colon M\to E_j$ be pointed maps. Suppose that there
are $C>0$ and positive numbers $(a_j)$ and $(b_j)$ such that
\begin{equation}\label{eq:lacunary-scales}
 b_j\leq\frac{a_j}{4},
 \qquad
 a_{j+1}\leq\frac{b_j}{4}
 \qquad(j\geq1),
\end{equation}
and
\begin{equation}\label{eq:lacunary-three-scale}
 \norm{\Phi_j(x)-\Phi_j(y)}^p
 \leq
 C\min\set{d(x,y)^p,d(x,y)^qb_j^{p-q},a_j^p}
 \qquad(x,y\in M).
\end{equation}
Put
\begin{equation}\label{eq:lacunary-Hp}
 H_{p,q}=
 \frac1{1-16^{p-q}}+1+\frac1{1-16^{-p}}.
\end{equation}
Then
\[
 \Phi(x)=(\Phi_j(x))_{j=1}^{\infty}
\]
defines a pointed Lipschitz map from $M$ into
$(\bigoplus_{j=1}^{\infty}E_j)_p$, and
\[
 \Lip(\Phi)^p\leq CH_{p,q}.
\]
\end{lemma}

\begin{proof}
The scale assumptions give
\[
 a_{j+1}\leq\frac{a_j}{16},
 \qquad
 b_{j+1}\leq\frac{b_j}{16}.
\]
In particular, both sequences converge to zero, and the intervals
$(b_j,a_j)$ are pairwise disjoint and ordered from right to left.

Fix distinct $x,y\in M$ and put $h=d(x,y)$. Split the indices into
\[
 I_1=\set{j:h\leq b_j},
 \qquad
 I_2=\set{j:b_j<h<a_j},
 \qquad
 I_3=\set{j:a_j\leq h}.
\]
The set $I_1$ is either empty or a finite initial segment. If
$I_1=\{1,\ldots,j_0\}$, then
\[
 b_j\geq16^{j_0-j}b_{j_0}\qquad(j\leq j_0).
\]
Since $p-q<0$ and $h\leq b_{j_0}$, it follows that
\[
 \sum_{j\in I_1}h^qb_j^{p-q}
 \leq
 h^qb_{j_0}^{p-q}
 \sum_{k=0}^{j_0-1}16^{k(p-q)}
 \leq
 \frac{h^qb_{j_0}^{p-q}}{1-16^{p-q}}
 \leq
 \frac{h^p}{1-16^{p-q}}.
\]

Since the intervals $(b_j,a_j)$ are pairwise disjoint, $I_2$ contains at
most one index, and hence
\[
 \sum_{j\in I_2}h^p\leq h^p.
\]

Finally, $I_3$ is a nonempty tail. If $j_1=\min I_3$, then
\[
 \sum_{j\in I_3}a_j^p
 \leq
 a_{j_1}^p\sum_{k=0}^{\infty}16^{-kp}
 \leq
 \frac{h^p}{1-16^{-p}}.
\]
Therefore, by \eqref{eq:lacunary-three-scale},
\begin{align*}
 \sum_{j=1}^{\infty}\norm{\Phi_j(x)-\Phi_j(y)}^p
 &\leq
 C\left(
   \sum_{j\in I_1}h^qb_j^{p-q}
   +\sum_{j\in I_2}h^p
   +\sum_{j\in I_3}a_j^p
 \right)
 \\
 &\leq CH_{p,q}h^p.
\end{align*}
Taking $y=0$ shows that $\Phi(x)$ belongs to
$(\bigoplus_{j=1}^{\infty}E_j)_p$, and the same estimate gives
$\Lip(\Phi)^p\leq CH_{p,q}$.
\end{proof}

\section{Quotient detectors and compact bases}
\label{sec:quotients}

We use the metric quotient from \cite[Section~2.2]{AACDEmbeddability}
and its universal property to identify the quotient of a free $p$-space
by the closed span of Dirac vectors over a closed subset.

\subsection{Quotients of \texorpdfstring{$p$}{p}-metric spaces}

Let $(M,d,0)$ be a pointed $p$-metric space and let $A\subseteq M$ be
closed and contain $0$. Following \cite[Section~2.2]{AACDEmbeddability},
let $M/A$ be the pointed $p$-metric space obtained by identifying all
points of $A$ with a single base point $*$. Write $\pi_A(x)=[x]$ for
the quotient map and $r_A(x)=\dist_d(x,A)$. Its distance is
\begin{equation}\label{eq:quotient-metric}
 d_A([x],[y])^p
 =\min\{d(x,y)^p,r_A(x)^p+r_A(y)^p\}.
\end{equation}
In particular, $d_A([x],*)=r_A(x)$ and $\pi_A$ is contractive.
By \cite[Proposition~2.6]{AACDEmbeddability}, for every pointed
$p$-metric space $Z$, each Lipschitz map $f\colon M\to Z$ vanishing
on $A$ factors uniquely as $f=g\pi_A$, where
$g\colon M/A\to Z$ is pointed and $\Lip(g)=\Lip(f)$.
Applied to $p$-Banach targets, this gives the following consequence.

\begin{lemma}[Free spaces over quotients]\label{lem:free-quotient}
The linearization of $\delta_{M/A}\pi_A$ induces an isometric
isomorphism
\[
 \Fp(M)/\Fp(A;M)\longrightarrow\Fp(M/A,d_A).
\]
\end{lemma}

\begin{proof}
Let $X=\Fp(M)$, let $Y=\Fp(A;M)$, and let $Q_A\colon X\to X/Y$
be the linear quotient map. Let $T=\widehat{\delta_{M/A}\pi_A}$. By the universal property,
$\norm T\leq1$. Since $T\delta_M(a)=0$ for every $a\in A$ and
$Y=\olsp\{\delta_M(a):a\in A\}$, continuity gives $Y\subseteq\ker T$.
Thus
\[
 U\colon X/Y\longrightarrow\Fp(M/A,d_A),\qquad U(u+Y)=Tu,
\]
is a well-defined linear map. For every $u\in X$ and $y\in Y$,
$\norm{U(u+Y)}=\norm{T(u-y)}\leq\norm{u-y}$. Taking the infimum
over $y\in Y$ gives
\[
 \norm{U(u+Y)}\leq\inf_{y\in Y}\norm{u-y}=\norm{u+Y}_{X/Y}.
\]
Hence $U$ is contractive. By construction, $UQ_A=T$, so
$UQ_A\delta_M=\delta_{M/A}\pi_A$.

Conversely, $Q_A\delta_M\colon M\to X/Y$ is contractive and vanishes
on $A$. The universal property from
\cite[Proposition~2.6]{AACDEmbeddability} gives a contraction
$g\colon M/A\to X/Y$ with $g\pi_A=Q_A\delta_M$.
Its linearization $V\colon\Fp(M/A,d_A)\to X/Y$ is a contraction.
For every $x\in M$ we have
\[
 VUQ_A\delta_M(x)=Q_A\delta_M(x),\qquad
 UV\delta_{M/A}([x])=\delta_{M/A}([x]).
\]
The vectors in these identities have dense linear spans in the
respective spaces. Hence $VU=\Id_{X/Y}$ and
$UV=\Id_{\Fp(M/A,d_A)}$. Thus $U$ and $V$ are inverse contractions,
so $U$ is an isometric isomorphism.
\end{proof}

\begin{lemma}[Comparison of quotient detectors]
\label{lem:quotient-comparison}
Let $(M,d,0)$ and $(N,\sigma,0)$ be pointed $p$-metric spaces,
let $C\subseteq M$ be closed and pointed, and let $S\subseteq M$
be finite and pointed. Write $\pi_C\colon M\to(M/C,d_C)$ for
the quotient map. Suppose that $f\colon M\to N$ is pointed
and Lipschitz and that $D>0$ satisfies
\begin{equation}\label{eq:quotient-comparison-hypothesis}
 d_C(\pi_C(x),\pi_C(y))
 \leq D\sigma(f(x),f(y))
 \qquad(x,y\in S).
\end{equation}
Writing $T_f=\widehat{\delta_Nf}$, we have
\begin{equation}\label{eq:quotient-comparison}
 \dist(\gamma,\Fp(C;M))\leq D\Gamma_p\norm{T_f\gamma}
 \qquad(\gamma\in\spn\delta_M(S)).
\end{equation}
In particular, let $A\subseteq M$ be closed and pointed, and let
$D\geq1$. If
\begin{equation}\label{eq:quotient-depth-comparison}
 \dist_d(x,C)\leq D\dist_d(x,A)\qquad(x\in S),
\end{equation}
then, for every $\gamma\in\spn\delta_M(S)$,
\[
 \dist(\gamma,\Fp(C;M))
 \leq D\Gamma_p\dist(\gamma,\Fp(A;M)).
\]
\end{lemma}

\begin{proof}
By \eqref{eq:quotient-comparison-hypothesis}, the map
\[
 g\colon f(S)\longrightarrow\Fp(M/C,d_C),
 \qquad g(f(x))=\delta_{M/C}(\pi_C(x)),
\]
is well defined, pointed, and $D$-Lipschitz. Indeed, if $f(x)=f(y)$,
then \eqref{eq:quotient-comparison-hypothesis} gives
$\pi_C(x)=\pi_C(y)$. Moreover, $g(f(0))=0$, and the Lipschitz
estimate follows from the isometry of $\delta_{M/C}$.

Write $\gamma=\sum_{x\in S}c_x\delta_M(x)$ and put
\[
 v=\sum_{x\in S}c_x\delta_{f(S)}(f(x))\in\Fp(f(S)).
\]
Its image under $\widehat g$ is
\[
 \widehat g v
 =\sum_{x\in S}c_x\delta_{M/C}(\pi_C(x))
 =\widehat{\delta_{M/C}\pi_C}\,\gamma.
\]
By Lemma~\ref{lem:free-quotient}, this vector has norm
$\dist(\gamma,\Fp(C;M))$. On the other hand, the canonical image
of $v$ in $\Fp(N)$ is $T_f\gamma$. Since
$\norm{\widehat g}\leq D$, Theorem~\ref{thm:canonical-input} gives
\[
 \dist(\gamma,\Fp(C;M))
 =\norm{\widehat g v}
 \leq D\norm{v}_{\Fp(f(S))}
 \leq D\Gamma_p\norm{T_f\gamma}.
\]

For the last assertion, let $\pi_A\colon M\to(M/A,d_A)$ be the
quotient map and write $r_A(x)=\dist_d(x,A)$ and
$r_C(x)=\dist_d(x,C)$. Since $D\geq1$,
\eqref{eq:quotient-depth-comparison} and the quotient formula give,
for $x,y\in S$,
\begin{align*}
 d_C(\pi_C(x),\pi_C(y))^p
 &=\min\{d(x,y)^p,r_C(x)^p+r_C(y)^p\}\\
 &\leq D^p\min\{d(x,y)^p,r_A(x)^p+r_A(y)^p\}\\
 &=D^p d_A(\pi_A(x),\pi_A(y))^p.
\end{align*}
Apply the first assertion with $(N,\sigma)=(M/A,d_A)$ and
$f=\pi_A$. By Lemma~\ref{lem:free-quotient}, we obtain
\[
 \dist(\gamma,\Fp(C;M))
 \leq D\Gamma_p\norm{\widehat{\delta_{M/A}\pi_A}\,\gamma}
 =D\Gamma_p\dist(\gamma,\Fp(A;M)).
 \qedhere
\]
\end{proof}

\subsection{Flattening and truncating the quotient metric}

Let $0<p<q\leq1$ and let $(M,d,0)$ be a pointed $q$-metric space.
For $b>0$, define
\begin{equation}\label{eq:flat-metric}
 \rho_b(x,y)^p
 =\min\set{d(x,y)^p,b^{p-q}d(x,y)^q}.
\end{equation}

\begin{lemma}\label{lem:rho-pmetric}
The function $\rho_b$ is a $p$-metric on $M$. More precisely,
\begin{equation}\label{eq:rho-formula}
 \rho_b(x,y)^p=
 \begin{cases}
  b^{p-q}d(x,y)^q,&d(x,y)\leq b,\\
  d(x,y)^p,&d(x,y)\geq b.
 \end{cases}
\end{equation}
Moreover,
\begin{equation}\label{eq:rho-comparison}
 \rho_b(x,y)\leq d(x,y),
 \qquad
 \rho_b(x,y)=d(x,y)\quad\text{if }d(x,y)\geq b.
\end{equation}
\end{lemma}

\begin{proof}
Put $f_b(t)=\min\set{t^{p/q},b^{p-q}t}$ for $t\geq0$. This function
is linear on $[0,b^q]$ and equals $t^{p/q}$ on $[b^q,\infty)$.
Its left derivative at $b^q$ is $b^{p-q}$, while its right derivative
is $(p/q)b^{p-q}\leq b^{p-q}$. Since the derivative of $t^{p/q}$
is positive and decreasing, $f_b$ is increasing and concave, with
$f_b(0)=0$. Every such function is subadditive: for $s,t>0$,
concavity gives
\[
 f_b(s)\geq\frac{s}{s+t}f_b(s+t),
 \qquad
 f_b(t)\geq\frac{t}{s+t}f_b(s+t),
\]
and adding yields $f_b(s+t)\leq f_b(s)+f_b(t)$. The same inequality
is immediate if $s=0$ or $t=0$. Since $d^q$ is a metric and
$f_b(t)>0$ for $t>0$, it follows that $f_b\circ d^q$ is a metric.
As $\rho_b^p=f_b\circ d^q$, the function $\rho_b$ is a $p$-metric.
The formulas follow by comparing $t^p$ with $b^{p-q}t^q$ on the
two sides of $b$.
\end{proof}

\begin{lemma}[Finite-support recovery modulo a finite set]
\label{lem:quotient-recovery}
Let $F\subseteq M$ be finite and pointed and let $a>0$. There are
$p$-Banach spaces $E_{F,a,b}$ and linear contractions
$T_{F,a,b}\colon\Fp(M)\to E_{F,a,b}$, $b>0$, annihilating
$\Fp(F;M)$, such that, with $\Phi_{F,a,b}=T_{F,a,b}\delta_M$,
\begin{equation}\label{eq:quotient-three-scale}
 \norm{\Phi_{F,a,b}(x)-\Phi_{F,a,b}(y)}^p
 \leq\min\{h^p,h^qb^{p-q},a^p\},
 \qquad h=d(x,y),
\end{equation}
for all $x,y\in M$. Moreover, for every
$\gamma\in\spn\delta_M([F]_{a/2^{1/p}})$ there exists
$b_0=b_0(F,a,\gamma)>0$ such that, for every
$0<b\leq b_0$,
\begin{equation}\label{eq:quotient-recovery}
 \norm{T_{F,a,b}\gamma}
 \geq\Gamma_p^{-1}\dist(\gamma,\Fp(F;M)).
\end{equation}
\end{lemma}

\begin{proof}
For each $b>0$, let $\rho_{b,F}$ be the quotient $p$-metric on
$M/F$ associated with $\rho_b$ by \eqref{eq:quotient-metric}, namely
\[
 \rho_{b,F}([x],[y])^p
 =\min\bigl\{\rho_b(x,y)^p,\,
   \dist_{\rho_b}(x,F)^p+\dist_{\rho_b}(y,F)^p\bigr\}.
\]
On the same set $M/F$, with base point $[0]=F$, define
\[
 \tau([x],[y])^p
 =\min\{\rho_{b,F}([x],[y])^p,a^p\}.
\]
Since $\rho_{b,F}^p$ is a metric and
\[
 \min\{u+v,c\}\leq\min\{u,c\}+\min\{v,c\}
 \qquad(u,v\geq0,\ c>0),
\]
the function $\tau^p$ is also a metric. We denote the resulting
pointed $p$-metric space $(M/F,\tau,[0])$ by $N_{F,a,b}$.

Put $r_F(x)=\dist_d(x,F)$ and
$\varphi_b(t)=\min\{t^p,b^{p-q}t^q\}$. Since $F$ is finite and
$\varphi_b$ is increasing,
\[
 \dist_{\rho_b}(x,F)^p
 =\min_{z\in F}\varphi_b(d(x,z))
 =\varphi_b(r_F(x)).
\]
Substituting this identity and
$\rho_b(x,y)^p=\varphi_b(d(x,y))$ into the definitions of
$\rho_{b,F}$ and $\tau$ gives
\begin{equation}\label{eq:capped-quotient}
 \tau([x],[y])^p
 =\min\{\varphi_b(d(x,y)),\,
          \varphi_b(r_F(x))+\varphi_b(r_F(y)),\,a^p\}.
\end{equation}

Let $\pi\colon M\to N_{F,a,b}$ be the quotient map and put
\[
 E_{F,a,b}=\Fp(N_{F,a,b}),\qquad
 \Phi_{F,a,b}=\delta_{N_{F,a,b}}\pi.
\]
Since $\delta_{N_{F,a,b}}$ is isometric, the first and last terms
in \eqref{eq:capped-quotient} give
\eqref{eq:quotient-three-scale}. In particular, $\Phi_{F,a,b}$
is contractive, so its linearization
$T_{F,a,b}=\widehat\Phi_{F,a,b}$ is a contraction satisfying
$T_{F,a,b}\delta_M=\Phi_{F,a,b}$. Since $\Phi_{F,a,b}$ vanishes
on $F$, linearity and continuity imply that $T_{F,a,b}$
annihilates $\Fp(F;M)$.

To prove \eqref{eq:quotient-recovery}, let
$S\subseteq[F]_{a/2^{1/p}}$ be a finite pointed support of $\gamma$.
Let $b_0$ be the minimum of the positive numbers in
\[
 \{d(x,y):x,y\in S,\ x\ne y\}
 \ \cup\ \{r_F(x):x\in S\setminus F\}.
\]
If this set is empty, put $b_0=1$. Pick $0<b\leq b_0$. Since $S\subseteq[F]_{a/2^{1/p}}$, for every $x\in S$ we have
$r_F(x)\leq a/2^{1/p}$, and hence for $x,y\in S$,
\[
 r_F(x)^p+r_F(y)^p
 \leq \frac{a^p}{2}+\frac{a^p}{2}=a^p.
\]
Moreover, $\varphi_b(t)=t^p$ whenever $t=0$ or $t\geq b$.
Our choice of $b$ therefore ensures that
\[
 \varphi_b(d(x,y))=d(x,y)^p,
 \qquad
 \varphi_b(r_F(x))=r_F(x)^p
 \qquad(x,y\in S).
\]

Let $\pi_F\colon M\to(M/F,d_F)$ denote the quotient map associated
with the original distance $d$. By \eqref{eq:capped-quotient},
for $x,y\in S$ we obtain
\begin{align*}
 \tau(\pi(x),\pi(y))^p
 &=\min\{d(x,y)^p,r_F(x)^p+r_F(y)^p,a^p\}\\
 &=\min\{d(x,y)^p,r_F(x)^p+r_F(y)^p\}\\
 &=d_F(\pi_F(x),\pi_F(y))^p.
\end{align*}
The second equality holds because the second term in the minimum
is already at most $a^p$; the last equality is the quotient formula
\eqref{eq:quotient-metric}.

The preceding equality verifies
\eqref{eq:quotient-comparison-hypothesis} with $C=F$,
$(N,\sigma)=(N_{F,a,b},\tau)$, $f=\pi$, and $D=1$.
Since $T_f=T_{F,a,b}$, Lemma~\ref{lem:quotient-comparison} gives
\[
 \dist(\gamma,\Fp(F;M))
 \leq\Gamma_p\norm{T_{F,a,b}\gamma},
\]
which is \eqref{eq:quotient-recovery}.
\end{proof}

\subsection{Kalton's property and quantitative selection}

We use the following extension of
\cite[Definition~5.2]{AABCSchur}. That definition is stated for metric
bases and includes a relative version; here we take the ambient base
itself and allow a $p$-metric.

\begin{definition}\label{def:Kalton}
A bounded set $W\subseteq\Fp(M)$ has \emph{Kalton's property} if for
every $r,\eta>0$ there is a finite pointed set $F\subseteq M$ such that
\[
 W\subseteq\Fp([F]_r;M)+\eta B_{\Fp(M)}.
\]
\end{definition}

\begin{theorem}[Quantitative selection under Kalton's property]
\label{thm:Kalton-selection}
Let $0<p<q\leq1$, let $M$ be a pointed $q$-metric space, and let
$W\subseteq\Fp(M)$ be bounded, infinite, and $\delta$-separated.
Suppose that $W$ has Kalton's property. For every
\[
 0<c<\frac{\delta}{2^{1/p}\Gamma_p H_{p,q}^{1/p}},
\]
there is a sequence $(w_j)$ in $W$ with lower $\ell_p$-estimate $c$.
\end{theorem}

\begin{proof}
Put $L=H_{p,q}^{1/p}$ and choose $\rho$ so that
\begin{equation}\label{eq:Kalton-rho}
 c\Gamma_p L<\rho<\delta/2^{1/p}.
\end{equation}
We first observe that, for every finite-dimensional subspace
$Y\subseteq\Fp(M)$, infinitely many $w\in W$ satisfy
$\dist(w,Y)>\rho$. Otherwise, choose
$\rho<s<\delta/2^{1/p}$. For all but finitely many $w\in W$ there
would be $y_w\in Y$ with $\norm{w-y_w}<s$. The family $(y_w)$ is
bounded, since $W$ is bounded and
$\norm{y_w}^p\leq\norm w^p+s^p$. It is therefore totally bounded.
There would be distinct $u,v\in W$ with
\[
 \norm{y_u-y_v}^p<\delta^p-2s^p.
\]
The $p$-triangle inequality would then give $\norm{u-v}^p<\delta^p$,
contradicting separation.

Choose $0<\varepsilon<\rho$ such that
\begin{equation}\label{eq:Kalton-epsilon}
 c^pL^p<\Gamma_p^{-p}(\rho^p-\varepsilon^p)-\varepsilon^p.
\end{equation}
This is possible by \eqref{eq:Kalton-rho}. We recursively select
scales $a_j,b_j>0$, finite pointed sets $F_j$, distinct $w_j\in W$,
and finitely supported vectors $\gamma_j$.

Choose $a_1>0$. At step $j>1$, choose $a_j\leq b_{j-1}/4$.
Kalton's property provides a finite pointed set $F_j$ such that every
$w\in W$ can be approximated within $\varepsilon$ by a finitely
supported vector supported in $[F_j]_{a_j/2^{1/p}}$. More explicitly,
choose $\eta>0$ with $2\eta^p<\varepsilon^p$, apply
Definition~\ref{def:Kalton} with this $\eta$ and radius $a_j/2^{1/p}$,
and then approximate the resulting vector in the closed span by a
finite linear combination of its Dirac vectors with error less than
$\eta$.

The finite-dimensional observation, applied to $Y_j=\Fp(F_j;M)$,
allows us to choose $w_j$ distinct from the previous choices with
$\dist(w_j,Y_j)>\rho$. Choose the prescribed approximation
$\gamma_j$ with $\norm{w_j-\gamma_j}<\varepsilon$ and finite pointed
support in $[F_j]_{a_j/2^{1/p}}$. In the quotient by $Y_j$,
\[
 \dist(\gamma_j,Y_j)^p
 \geq\dist(w_j,Y_j)^p-\norm{w_j-\gamma_j}^p
 >\rho^p-\varepsilon^p.
\]
By Lemma~\ref{lem:quotient-recovery}, choose $0<b_j\leq a_j/4$
such that, with $T_j=T_{F_j,a_j,b_j}$,
$\norm{T_j\gamma_j}\geq\Gamma_p^{-1}\dist(\gamma_j,Y_j)$.
Since $\norm{T_j}\leq1$, it follows that
\begin{equation}\label{eq:Kalton-diagonal}
 \norm{T_jw_j}^p
 \geq\norm{T_j\gamma_j}^p-\norm{w_j-\gamma_j}^p
 >\Gamma_p^{-p}(\rho^p-\varepsilon^p)-\varepsilon^p
 =:\alpha^p>c^pL^p.
\end{equation}

The scales satisfy \eqref{eq:lacunary-scales}. By
\eqref{eq:quotient-three-scale} and Lemma~\ref{lem:lacunary-assembly},
the map $(\Phi_{F_j,a_j,b_j})_j$ linearizes to an operator
\[
 T\colon\Fp(M)\longrightarrow
 \left(\bigoplus_{j=1}^{\infty}E_{F_j,a_j,b_j}\right)_p,
 \qquad \norm T\leq L,
\]
whose $j$th coordinate is $T_j$. Since $c<\alpha/L$,
Lemma~\ref{lem:diagonal-detector}, applied directly to $(w_j)$,
gives a subsequence with the required lower $\ell_p$-estimate $c$.
\end{proof}

\begin{theorem}[Compact bases]\label{thm:compact-strong}
Let $0<p<q\leq1$ and let $K$ be a compact pointed $q$-metric
space. Then $\Fp(K)$ has the $K_{p,q}$-strong Schur $p$-property, where
\begin{equation}\label{eq:compact-constant}
 K_{p,q}=\Gamma_p H_{p,q}^{1/p}
 =\Gamma_p\left(\frac1{1-16^{p-q}}+1+
                         \frac1{1-16^{-p}}\right)^{1/p}.
\end{equation}
In particular, every compact metric space has the same bound $K_{p,1}$.
\end{theorem}

\begin{proof}
Every bounded subset of $\Fp(K)$ has Kalton's property. Indeed, for
each $r>0$, a finite pointed $r$-net $F\subseteq K$ satisfies
$[F]_r=K$, so the inclusion in Definition~\ref{def:Kalton} holds for
every $\eta>0$.

Let $\delta>0$, let $K_0>K_{p,q}$, and let $W\subseteq S_{\Fp(K)}$
be infinite and $\delta$-separated. Since
\[
 \frac{\delta}{K_0 2^{1/p}}
 <\frac{\delta}{2^{1/p}\Gamma_p H_{p,q}^{1/p}},
\]
Theorem~\ref{thm:Kalton-selection} supplies a sequence in $W$ with
lower estimate $\delta/(K_0 2^{1/p})$. Its upper estimate is one by
the $p$-triangle inequality. This is the required strong Schur
$p$-property.
\end{proof}

\begin{corollary}[Finite-dimensional subsets]
\label{thm:finite-dimensional-subsets}
Let $X$ be a finite-dimensional normed space and let $M\subseteq X$.
Then $\Fp(M)$ has the strong Schur $p$-property.
\end{corollary}

\begin{proof}
The unit ball $B_X$ is compact, so $\Fp(B_X)$ has the strong Schur
$p$-property by Theorem~\ref{thm:compact-strong} with $q=1$.
By \cite[Theorem~4.15]{AACDInfiniteSums}, applied with its cone equal
to the Banach space $X$ and its snowflake exponent equal to $1$,
\[
 \Fp(X)\simeq\Fp(B_X).
\]
Thus $\Fp(X)$ has the strong Schur $p$-property by isomorphic
invariance. After translating a chosen base point of $M$ to $0$,
Theorem~\ref{thm:canonical-input} identifies $\Fp(M)$ with a closed
subspace of $\Fp(X)$. The property passes to this subspace and then
to $\Fp(M)$ by isomorphic invariance.
\end{proof}
\section{Bounded bases and the Kalton--\texorpdfstring{$\ell_p$}{ell-p} alternative}
\label{sec:bounded}

We first isolate a localization principle for arbitrary $p$-metric
bases. It will be used both when Kalton's property fails on a
bounded base and in the localization at infinity in
Section~\ref{sec:general}. The geometric hypothesis compares
distances to two closed sets on the finite supports; the complements
of one family must be pairwise disjoint.

\begin{lemma}[Localization by disjoint quotients]
\label{lem:disjoint-localization}
Let $M$ be a pointed $p$-metric space. For each $j\in\NN$, let
$A_j,C_j\subseteq M$ be closed and pointed and let $S_j\subseteq M$
be finite and pointed. Suppose that the sets $M\setminus A_j$ are
pairwise disjoint and, for some $D\geq1$,
\begin{equation}\label{eq:disjoint-depth}
 \dist_d(x,C_j)\leq D\dist_d(x,A_j)
 \qquad(j\in\NN,\ x\in S_j).
\end{equation}
If $(w_j)$ is a bounded sequence in $\Fp(M)$ with no subsequence
equivalent to the canonical basis of $\ell_p$, then
\[
 \dist(w_j,\Fp(S_j;M))\longrightarrow0
 \quad\Longrightarrow\quad
 \dist(w_j,\Fp(C_j;M))\longrightarrow0.
\]
\end{lemma}

\begin{proof}
Put $E_j=\Fp(M/A_j,d_{A_j})$ and
$\Phi_j=\delta_{M/A_j}\pi_{A_j}$. Each $\Phi_j$ is pointed and
contractive. Since the complements of the $A_j$ are pairwise
disjoint, at each $x\in M$ at most one $\Phi_j(x)$ is nonzero.
Thus $\Phi(x)=(\Phi_j(x))_j$ belongs to $(\bigoplus_jE_j)_p$ and
\[
 \norm{\Phi(x)-\Phi(y)}^p\leq2d(x,y)^p\qquad(x,y\in M),
\]
because the difference has at most two nonzero coordinates.
Its linearization $T\colon\Fp(M)\to(\bigoplus_jE_j)_p$ has
$\norm T\leq2^{1/p}$ and coordinate operators
$T_j=\widehat\Phi_j$, each of norm at most one.
Lemma~\ref{lem:diagonal-detector} implies
$\norm{T_jw_j}\to0$, since otherwise $(w_j)$ would have an
$\ell_p$-subsequence.

Choose $\gamma_j\in\Fp(S_j;M)=\spn\delta_M(S_j)$ with
$\norm{w_j-\gamma_j}\to0$. Then
\[
 \norm{T_j\gamma_j}^p
 \leq\norm{T_jw_j}^p+\norm{w_j-\gamma_j}^p\longrightarrow0.
\]
For each $j$, condition \eqref{eq:disjoint-depth} is precisely
\eqref{eq:quotient-depth-comparison} with $A=A_j$, $C=C_j$,
and $S=S_j$. Thus the last assertion of
Lemma~\ref{lem:quotient-comparison} gives
\begin{align*}
 \dist(\gamma_j,\Fp(C_j;M))
 &\leq D\Gamma_p\dist(\gamma_j,\Fp(A_j;M))\\
 &=D\Gamma_p\norm{T_j\gamma_j}\longrightarrow0,
\end{align*}
where the equality follows from Lemma~\ref{lem:free-quotient}.
Finally, the quotient $p$-triangle inequality gives
\[
 \dist(w_j,\Fp(C_j;M))^p
 \leq\norm{w_j-\gamma_j}^p
     +\dist(\gamma_j,\Fp(C_j;M))^p\longrightarrow0.
 \qedhere
\]
\end{proof}

We now apply this principle when Kalton's property fails on a
bounded $p$-metric base.

\begin{lemma}\label{lem:non-Kalton}
Let $M$ be a bounded pointed $p$-metric space. If a bounded set
$W\subseteq\Fp(M)$ fails Kalton's property, then $W$ contains a
sequence equivalent to the canonical basis of $\ell_p$.
\end{lemma}

\begin{proof}
Suppose, towards a contradiction, that $W$ contains no sequence
equivalent to the canonical basis of $\ell_p$. Put
$R=\operatorname{diam}M$. Failure of Kalton's property provides
$r,\eta>0$ such that, for every finite pointed $F\subseteq M$,
some $w\in W$ satisfies
\begin{equation}\label{eq:non-Kalton-failure}
 \dist(w,\Fp([F]_r;M))>\eta.
\end{equation}
Necessarily $R>0$, and we may take $r<R$, since
$[\{0\}]_r=M$ when $r\geq R$. Set $s=r/2^{1/p}$ and $D=R/s$.

Starting with $F_0=\{0\}$, recursively set $C_j=[F_{j-1}]_r$,
choose $w_j\in W$ with $\dist(w_j,\Fp(C_j;M))>\eta$, and choose
a finitely supported $\gamma_j$ with
$\norm{w_j-\gamma_j}<1/j$. Let $F_j$ be the union of $F_{j-1}$
and a finite pointed support of $\gamma_j$, and put
\begin{equation}\label{eq:non-Kalton-A}
 A_j=[F_{j-1}]_s
       \ \cup\ \{x\in M:\dist_d(x,F_j)\geq s\}.
\end{equation}
The sets $A_j$ and $C_j$ are closed and pointed. The sets $M\setminus A_j$ are pairwise disjoint. Indeed, fix
distinct indices $i,j$. Interchanging them if necessary, assume
$i<j$. If $x\in M\setminus A_i$, the definition of $A_i$ gives
$\dist_d(x,F_i)<s$. Since the sets $F_k$ are increasing and
$i\leq j-1$, we have
\[
 \dist_d(x,F_{j-1})\leq\dist_d(x,F_i)<s.
\]
Thus $x\in[F_{j-1}]_s\subseteq A_j$, so
$(M\setminus A_i)\cap(M\setminus A_j)=\varnothing$.

We verify \eqref{eq:disjoint-depth} with $S_j=F_j$ and $D=R/s$.
If $x\in F_j\setminus C_j$, then $\dist_d(x,A_j)\geq s$.
To see this, let $z\in[F_{j-1}]_s$. Since $F_{j-1}$ is finite,
choose $a\in F_{j-1}$ with $d(z,a)\leq s$. Then
\[
 d(x,z)^p\geq d(x,a)^p-d(z,a)^p
 \geq\dist_d(x,F_{j-1})^p-s^p
 >r^p-s^p=s^p.
\]
If instead $\dist_d(z,F_j)\geq s$, the fact that $x\in F_j$
gives $d(x,z)\geq s$. Taking the infimum over both parts of $A_j$
proves the assertion. Consequently, for $x\in F_j\setminus C_j$,
\[
 \dist_d(x,C_j)\leq d(x,0)\leq R
 \leq D\dist_d(x,A_j).
\]
For $x\in F_j\cap C_j$ the same inequality holds because its
left-hand side is zero.

Since $\gamma_j\in\Fp(F_j;M)$, we have
\[
 \dist(w_j,\Fp(F_j;M))
 \leq\norm{w_j-\gamma_j}<1/j\longrightarrow0.
\]
Moreover, $(w_j)$ is bounded and has no $\ell_p$-subsequence,
because its terms belong to $W$, which was assumed to contain
no such sequence. Having verified the remaining hypotheses,
we may apply Lemma~\ref{lem:disjoint-localization} with $S_j=F_j$.
It follows that $\dist(w_j,\Fp(C_j;M))\to0$. On the other hand, each $w_j$ was chosen so that
$\dist(w_j,\Fp(C_j;M))>\eta$. Hence
\[
 \liminf_{j\to\infty}\dist(w_j,\Fp(C_j;M))
 \geq\eta>0,
\]
a contradiction.
\end{proof}

\begin{theorem}[Bounded bases]\label{thm:bounded-Schur}
Let $0<p<q\leq1$ and let $M$ be a bounded pointed $q$-metric space.
Then $\Fp(M)$ has the Schur $p$-property.
\end{theorem}

\begin{proof}
Let $W\subseteq\Fp(M)$ be bounded, infinite, and uniformly separated.
If $W$ has Kalton's property, apply Theorem~\ref{thm:Kalton-selection}.
If it does not, apply Lemma~\ref{lem:non-Kalton}. The latter lemma
applies because every $q$-metric is a $p$-metric when $p<q$.
\end{proof}

\section{Arbitrary \texorpdfstring{$q$}{q}-metric spaces}
\label{sec:general}

The remaining task is to remove boundedness. We first prove a localization
lemma whose constants are independent of the radii of the selected
supports. For a pointed $p$-metric space $M$, put
\[
 B_M(R)=\{x\in M:d(x,0)\leq R\}\qquad(R>0).
\]

\subsection{Localization at infinity}

\begin{lemma}\label{lem:localization-infinity}
Let $M$ be a pointed $p$-metric space and let $W\subseteq\Fp(M)$ be
bounded. If $W$ contains no sequence equivalent to the canonical basis
of $\ell_p$, then for every $\varepsilon>0$ there is $R>0$ such that
\begin{equation}\label{eq:localization-infinity}
 W\subseteq\Fp(B_M(R);M)+\varepsilon B_{\Fp(M)}.
\end{equation}
\end{lemma}

\begin{proof}
Suppose that \eqref{eq:localization-infinity} fails. Choose $\eta>0$ such that for every $R>0$
some $w\in W$ satisfies
\begin{equation}\label{eq:infinity-failure}
 \dist(w,\Fp(B_M(R);M))>\eta.
\end{equation}
We recursively select radii $r_j,s_j>0$, vectors $w_j\in W$,
finite pointed sets $S_j$, and vectors
$\gamma_j\in\spn\delta_M(S_j)$. Choose $r_1>0$ and, at step
$j>1$, choose $r_j>s_{j-1}$. Put $C_j=B_M(2^{1/p}r_j)$, choose
$w_j$ with $\dist(w_j,\Fp(C_j;M))>\eta$, and choose $S_j$ and
$\gamma_j$ with $\norm{w_j-\gamma_j}<1/j$. Choose $s_j$ so that
\begin{equation}\label{eq:infinity-radii}
 s_j>2^{1/p}r_j,
 \qquad s_j^p>2\max_{x\in S_j}d(x,0)^p,
\end{equation}
and define
\[
 A_j=B_M(r_j)\ \cup\ \{x\in M:d(x,0)\geq s_j\}.
\]
The sets $A_j$ and $C_j$ are closed and pointed, and the complements
$M\setminus A_j=\{x:r_j<d(x,0)<s_j\}$ are pairwise disjoint.

We verify \eqref{eq:disjoint-depth} with $D=2^{1/p}$. If
$x\in S_j\setminus C_j$ and $t=d(x,0)$, then
\begin{equation}\label{eq:infinity-depth}
 \dist_d(x,A_j)^p\geq t^p/2.
\end{equation}
Indeed, for $z\in B_M(r_j)$, the reverse $p$-triangle inequality gives
\[
 d(x,z)^p\geq t^p-d(z,0)^p\geq t^p-r_j^p>t^p/2,
\]
because $t>2^{1/p}r_j$. For $d(z,0)\geq s_j$, it gives
\[
 d(x,z)^p\geq d(z,0)^p-t^p\geq s_j^p-t^p>t^p
\]
by \eqref{eq:infinity-radii}. Taking the infimum proves
\eqref{eq:infinity-depth}. Since $0\in C_j$, we obtain
\[
 \dist_d(x,C_j)\leq t\leq2^{1/p}\dist_d(x,A_j).
\]
For $x\in S_j\cap C_j$ this inequality holds because its
left-hand side is zero.

Now $\dist(w_j,\Fp(S_j;M))\leq\norm{w_j-\gamma_j}\to0$.
Lemma~\ref{lem:disjoint-localization} yields
$\dist(w_j,\Fp(C_j;M))\to0$, contradicting the choice of $w_j$.
\end{proof}

\subsection{Clipping in a universal \texorpdfstring{$q$}{q}-metric space}

\begin{lemma}\label{lem:ellinfty-q}
Let $I$ be a set and let $0<p<q\leq1$. Equip $\ell_\infty(I)$ with
the pointed $q$-metric
\[
 d_q(x,y)=\norm{x-y}_\infty^{1/q}.
\]
Then $\Fp(\ell_\infty(I),d_q)$ has the Schur $p$-property.
\end{lemma}

\begin{proof}
Put $\Omega=(\ell_\infty(I),d_q)$ and $X=\Fp(\Omega)$. For $R>0$,
let
\[
 \Omega_R=B_\Omega(R)
 =\{x\in\ell_\infty(I):\norm{x}_\infty\leq R^q\}.
\]
Define the coordinatewise clipping map $c_R\colon\Omega\to\Omega_R$ as
\[
 (c_Rx)(i)=\max\{-R^q,\min\{x(i),R^q\}\}\qquad(i\in I).
\]
The scalar clipping function is $1$-Lipschitz for the ordinary absolute
value. Hence $c_R$ is contractive for the supremum norm and also for
$d_q$. It is a pointed retraction onto $\Omega_R$, which is a bounded $q$-metric space. Let
$P_R:X\to X$ be the linearization of $\delta_\Omega c_R$.
It is a contractive projection with range $\Fp(\Omega_R;\Omega)$
and it fixes that range pointwise. The canonical inclusion of
$\Fp(\Omega_R)$ into $X$ is isometric: its contractive left inverse
is the linearization of $c_R:\Omega\to\Omega_R$. Thus the range of
$P_R$ has the Schur $p$-property by Theorem~\ref{thm:bounded-Schur}.

Let $W\subseteq X$ be bounded and contain no $\ell_p$-sequence. We
show that $W$ is totally bounded. For each $R$, the set $P_R(W)$ is
totally bounded. Otherwise it would contain an infinite uniformly
separated sequence, and the Schur $p$-property of the range of $P_R$
would give a sequence $(w_j)$ in $W$ such that $(P_Rw_j)$ has a
positive lower $\ell_p$-estimate. Since $P_R$ is contractive, the
same lower estimate holds for $(w_j)$. Its upper estimate follows from
boundedness, contradicting our assumption on $W$.

Given $\varepsilon>0$, Lemma~\ref{lem:localization-infinity} provides
$R$ such that every $w\in W$ can be written as $w=u+e$, where
$u\in\Fp(\Omega_R;\Omega)$ and $\norm e\leq\varepsilon$.
Because $P_Ru=u$,
\begin{equation}\label{eq:clipping-uniform}
 \norm{w-P_Rw}^p
 =\norm{e-P_Re}^p\leq2\norm e^p\leq2\varepsilon^p.
\end{equation}
To see total boundedness explicitly, let $t>0$ and choose
$\varepsilon>0$ with $3\varepsilon^p<t^p$. A finite
$\varepsilon$-net $\{v_1,\ldots,v_m\}$ for $P_R(W)$ then satisfies,
for each $w\in W$ and a suitable $k$,
\[
 \norm{w-v_k}^p
 \leq\norm{w-P_Rw}^p+\norm{P_Rw-v_k}^p
 \leq3\varepsilon^p<t^p.
\]
Thus $W$ is totally bounded. Consequently, every bounded subset of $X$
which is not totally bounded contains an $\ell_p$-sequence. In
particular, $X$ has the Schur $p$-property.
\end{proof}

\begin{theorem}\label{thm:main}
Let $0<p<q\leq1$ and let $(M,d,0)$ be a pointed $q$-metric space.
Then $\Fp(M,d)$ has the Schur $p$-property.
\end{theorem}

\begin{proof}
The ordinary metric $D=d^q$ admits the following pointed isometric
embedding into $\ell_\infty(M)$:
\[
 j(x)(z)=D(x,z)-D(0,z)\qquad(x,z\in M).
\]
Consequently $j$ is a pointed isometry of $(M,d)$ into
$(\ell_\infty(M),\norm{\cdot}_\infty^{1/q})$.
Theorem~\ref{thm:canonical-input} identifies $\Fp(M)$ isomorphically
with the closed span of its Dirac vectors in
$\Fp(\ell_\infty(M),\norm{\cdot}_\infty^{1/q})$. The latter space
has the Schur $p$-property by Lemma~\ref{lem:ellinfty-q}. Apply
Lemma~\ref{lem:Schur-stability}.
\end{proof}

\section{Applications}
\label{sec:consequences}
\label{sec:applications}

The universal Schur $p$-property has consequences for the linear and
nonlinear geometry of Lipschitz-free $p$-spaces.  We collect the most direct
ones in this section.  Besides giving structural information on
$\Fp(M)$ itself, they exhibit a sharp contrast between the linear and the
Lipschitz structures of nonlocally convex spaces.

Unless another assumption is explicitly stated, $M$ in this section
is a pointed $q_0$-metric space for some $p<q_0\leq1$. The exponent $q$
in the higher-convexity consequences may be any number with $p<q\leq1$;
it need not equal $q_0$.

\subsection{Saturation, subsequences, and higher-convexity obstructions}

\begin{lemma}\label{lem:Schur-saturated}
If a $p$-Banach space $X$ has the Schur $p$-property, then $X$ is
$\ell_p$-saturated.
\end{lemma}

\begin{proof}
Let $Y\subseteq X$ be an infinite-dimensional closed subspace.
By \cite[Proposition~2.6]{AABCSchur}, there is an infinite
$1/2$-separated set $A\subseteq S_Y$. The Schur $p$-property of
$X$ gives a sequence in $A$ equivalent to the canonical basis
of $\ell_p$. Its closed linear span is contained in $Y$ and
is isomorphic to $\ell_p$. Since $Y$ was arbitrary, $X$ is
$\ell_p$-saturated.
\end{proof}

\begin{corollary}\label{cor:saturation}
For every $M$ as above, the space $\Fp(M)$ is $\ell_p$-saturated.
\end{corollary}

\begin{proof}
Combine Theorem~\ref{thm:main} with Lemma~\ref{lem:Schur-saturated}.
\end{proof}

We next record an elementary obstruction which will be used repeatedly.

\begin{lemma}\label{lem:ellp-not-qBanach}
Let $p<q\leq1$. Then $\ell_p$ is not isomorphic to a $q$-Banach space.
\end{lemma}

\begin{proof}
Suppose that $\vertiii{\cdot}$ is an equivalent $q$-norm on $\ell_p$.
Choose $c,C>0$ such that
\[
 c\norm{x}_p\leq\vertiii{x}\leq C\norm{x}_p
 \qquad(x\in\ell_p).
\]
For the first $N$ unit vectors, $q$-subadditivity gives
\[
 cN^{1/p}
 =c\norm{\sum_{j=1}^Ne_j}_p
 \leq\vertiii{\sum_{j=1}^Ne_j}
 \leq\left(\sum_{j=1}^N\vertiii{e_j}^q\right)^{1/q}
 \leq CN^{1/q}.
\]
This is impossible for large $N$, because $1/p>1/q$.
\end{proof}

\begin{corollary}\label{cor:no-q-Banach}\label{cor:no-Banach}
Let $M$ be as above and let $p<q\leq1$. Then $\Fp(M)$ contains no
isomorphic copy of an infinite-dimensional $q$-Banach space. In particular,
it contains no isomorphic copy of an infinite-dimensional Banach space.
\end{corollary}

\begin{proof}
Suppose that an infinite-dimensional closed subspace $Y\subseteq\Fp(M)$ is
isomorphic to a $q$-Banach space. By Corollary~\ref{cor:saturation}, $Y$
contains a closed subspace isomorphic to $\ell_p$. Every closed subspace of
a $q$-Banach space is again a $q$-Banach space, contradicting
Lemma~\ref{lem:ellp-not-qBanach}.
\end{proof}

The defining selection property also admits a useful sequential
reformulation.

\begin{corollary}[Subsequence dichotomy]\label{cor:subsequence-dichotomy}
Every bounded sequence in $\Fp(M)$ has either a norm-convergent subsequence
or a subsequence equivalent to the canonical basis of $\ell_p$.
\end{corollary}

\begin{proof}
Let $(u_n)$ be bounded. If its range is relatively compact, then $(u_n)$ has
a norm-convergent subsequence. Otherwise its range is not totally bounded,
so there are $\varepsilon>0$ and a subsequence $(u_{n_k})$ such that
\[
 \norm{u_{n_k}-u_{n_l}}\geq\varepsilon
 \qquad(k\ne l).
\]
Theorem~\ref{thm:main} applied to the set $\{u_{n_k}:k\in\NN\}$ yields a
further subsequence equivalent to the canonical basis of $\ell_p$.
\end{proof}

\subsection{An operator dichotomy}

We say that a bounded linear operator $T:E\to Y$ \emph{fixes a copy of
$\ell_p$} if there is a closed subspace $E_0\subseteq E$, isomorphic to
$\ell_p$, such that $T|_{E_0}$ is an isomorphic embedding.

\begin{proposition}[Compact-or-$\ell_p$ dichotomy]
\label{prop:operator-dichotomy}
Let $E$ be a $p$-Banach space, let $M$ be as above, and let
$T:E\to\Fp(M)$ be bounded and linear. Then exactly one of the following
alternatives holds:
\begin{enumerate}[label=\textup{(\roman*)}]
\item $T$ is compact;
\item $T$ fixes a copy of $\ell_p$.
\end{enumerate}
\end{proposition}

\begin{proof}
Assume that $T$ is not compact. Then $T(B_E)$ is not relatively compact and,
since $\Fp(M)$ is complete, it is not totally bounded. Consequently, there
are $\varepsilon>0$ and $(x_n)\subseteq B_E$ such that
\[
 \norm{Tx_n-Tx_m}\geq\varepsilon
 \qquad(n\ne m).
\]
By Theorem~\ref{thm:main}, after passing to a subsequence there is a constant
$c>0$ such that
\begin{equation}\label{eq:operator-lower-image}
 c\left(\sum_n|a_n|^p\right)^{1/p}
 \leq\norm{\sum_na_nTx_n}
\end{equation}
for every finitely supported scalar sequence $(a_n)$. On the other hand,
$p$-subadditivity and $x_n\in B_E$ give
\begin{equation}\label{eq:operator-upper-domain}
 \norm{\sum_na_nx_n}^p
 \leq\sum_n|a_n|^p\norm{x_n}^p
 \leq\sum_n|a_n|^p.
\end{equation}
Since
\[
 \norm{\sum_na_nTx_n}
 \leq\norm T\,\norm{\sum_na_nx_n},
\]
\eqref{eq:operator-lower-image} and
\eqref{eq:operator-upper-domain} show that $(x_n)$ is equivalent to the
canonical basis of $\ell_p$. Moreover,
\[
 \norm{T\sum_na_nx_n}
 \geq c\left(\sum_n|a_n|^p\right)^{1/p}
 \geq c\norm{\sum_na_nx_n},
\]
so $T$ is bounded below on $\olsp\{x_n:n\in\NN\}$. Thus $T$ fixes a copy of
$\ell_p$.

Conversely, a compact operator cannot be bounded below on an
infinite-dimensional closed subspace. Hence the alternatives are mutually
exclusive.
\end{proof}

\begin{corollary}\label{cor:q-to-free-compact}
Let $p<q\leq1$, let $E$ be a $q$-Banach space, and let $M$ be as above. Then every bounded linear operator $E\to\Fp(M)$ is compact.
\end{corollary}

\begin{proof}
If an operator $E\to\Fp(M)$ were noncompact, then
Proposition~\ref{prop:operator-dichotomy} would produce a copy of $\ell_p$
in $E$, contrary to Lemma~\ref{lem:ellp-not-qBanach}.
\end{proof}

\subsection{The Albiac--Kalton lifting prediction}

For a $p$-Banach space $X$, Albiac and Kalton's $p$-Lipschitz lifting
property is equivalent to the existence of a bounded linear right inverse
of the barycenter map $\beta_X:\Fp(X)\to X$; this is the equivalence
\textup{(i)}$\Leftrightarrow$\textup{(ii)} in
\cite[Proposition~4.5]{AlbiacKalton}, where the free space is denoted
by $\text{\AE}_p(X)$.

\begin{corollary}\label{cor:AK-lifting}
Let $0<p<q\leq1$. No infinite-dimensional $q$-Banach space has the
$p$-Lipschitz lifting property.
\end{corollary}

\begin{proof}
Let $X$ be an infinite-dimensional $q$-Banach space. Its quasi-norm
distance $d(x,y)=\norm{x-y}$ is a $q$-metric, and $X$ is also a
$p$-Banach space. The identity map on $X$ has the contractive
linearization
\[
 \beta_X:\Fp(X,d)\longrightarrow X,
 \qquad \beta_X\delta_X(x)=x.
\]
If $X$ had the $p$-Lipschitz lifting property, the cited proposition would
give a bounded linear $S:X\to\Fp(X,d)$ with $\beta_XS=\Id_X$.
For every $x\in X$,
\[
 \norm x=\norm{\beta_XSx}\leq\norm{Sx}\leq\norm S\,\norm x.
\]
Thus $S$ would embed $X$ isomorphically into $\Fp(X,d)$, contrary to
Corollary~\ref{cor:no-q-Banach}.
\end{proof}

This proves the prediction in \cite[p.~334]{AlbiacKalton} throughout
the stated range $p<q\leq1$, with no separability assumption on $X$.

\subsection{A universal separable space with no Banach subspaces}

\begin{corollary}\label{cor:universal-Fpc0}
The separable $p$-Banach space
\[
 \mathbb U_p=\Fp(c_0)
\]
is bi-Lipschitz universal for all separable metric spaces. Nevertheless,
for every $p<q\leq1$ it contains no infinite-dimensional $q$-Banach subspace, and
every bounded linear operator from a $q$-Banach space into $\mathbb U_p$ is
compact.
\end{corollary}

\begin{proof}
The space $\Fp(c_0)$ is separable because $c_0$ is separable. Let $S$ be a
separable pointed metric space. By Aharoni's theorem \cite{Aharoni}, there is
a bi-Lipschitz embedding $f:S\to c_0$. Replacing $f(x)$ by
$f(x)-f(0)$, we may assume that $f$ is pointed. Since the canonical map
$\delta_{c_0}:c_0\to\Fp(c_0)$ is isometric, the composition
$\delta_{c_0}f$ is a bi-Lipschitz embedding of $S$ into $\mathbb U_p$.
The linear assertions follow from Corollaries~\ref{cor:no-q-Banach} and
\ref{cor:q-to-free-compact}.
\end{proof}

\subsection{Lipschitz structure does not determine linear structure}

Let $p<q\leq1$ and let $X$ be a $q$-Banach space, equipped with its
quasi-norm distance. The identity map on $X$ linearizes to the
contractive barycenter map (of norm one when $X\ne\{0\}$)
\begin{equation}\label{eq:barycenter-map}
 \beta_X:\Fp(X)\longrightarrow X,
 \qquad
 \beta_X\delta_X(x)=x.
\end{equation}
Put
\[
 Z_X=\ker\beta_X.
\]
We equip $X\oplus Z_X$ with the $p$-sum quasi-norm
\[
 \norm{(x,z)}_{X\oplus_pZ_X}
 =\bigl(\norm x^p+\norm z^p\bigr)^{1/p}.
\]

\begin{proposition}[Nonlinear splitting]\label{prop:nonlinear-splitting}
For every $q$-Banach space $X$ with $p<q\leq1$, the maps
\begin{align*}
 \Phi_X:\Fp(X)&\longrightarrow X\oplus_pZ_X,
 &\Phi_X(\mu)&=
 \bigl(\beta_X\mu,\mu-\delta_X(\beta_X\mu)\bigr),\\
 \Psi_X:X\oplus_pZ_X&\longrightarrow\Fp(X),
 &\Psi_X(x,z)&=\delta_X(x)+z
\end{align*}
are mutually inverse bi-Lipschitz maps. More precisely,
\[
 \Lip(\Psi_X)\leq1,
 \qquad
 \Lip(\Phi_X)\leq3^{1/p}.
\]
\end{proposition}

\begin{proof}
Since $\beta_X\delta_X=\Id_X$, the second coordinate of $\Phi_X(\mu)$
belongs to $Z_X$. Direct computation gives
\[
 \Psi_X\Phi_X(\mu)=\mu
 \qquad\text{and}\qquad
 \Phi_X\Psi_X(x,z)=(x,z).
\]
For $(x,z),(y,w)\in X\oplus_pZ_X$, the $p$-triangle inequality and the
isometry of $\delta_X$ give
\[
 \norm{\Psi_X(x,z)-\Psi_X(y,w)}^p
 \leq\norm{x-y}^p+\norm{z-w}^p.
\]
Thus $\Lip(\Psi_X)\leq1$.

Now let $\mu,\nu\in\Fp(X)$ and put $u=\mu-\nu$. Contractivity of $\beta_X$
and the isometry of $\delta_X$ yield
\begin{align*}
 &\norm{\Phi_X(\mu)-\Phi_X(\nu)}^p\\
 &\quad=
 \norm{\beta_Xu}^p
 +\norm{u-\bigl(\delta_X(\beta_X\mu)
                 -\delta_X(\beta_X\nu)\bigr)}^p\\
 &\quad\leq
 \norm u^p+\norm u^p+\norm{\beta_Xu}^p
 \leq3\norm u^p.
\end{align*}
Hence $\Lip(\Phi_X)\leq3^{1/p}$.
\end{proof}

\begin{corollary}\label{cor:not-Lipschitz-determined}
For every infinite-dimensional $q$-Banach space $X$ with $p<q\leq1$,
\[
 \Fp(X)\simeq_{\mathrm{Lip}}X\oplus_pZ_X,
\]
but the two $p$-Banach spaces are not linearly isomorphic.
\end{corollary}

\begin{proof}
The Lipschitz equivalence is
Proposition~\ref{prop:nonlinear-splitting}. The space $X\oplus_pZ_X$
contains the $1$-complemented $q$-Banach subspace $X\oplus\{0\}$, whereas
Corollary~\ref{cor:no-q-Banach} says that $\Fp(X)$ contains no
infinite-dimensional $q$-Banach subspace.
\end{proof}

The splitting in Proposition~\ref{prop:nonlinear-splitting} is the mechanism
behind the original examples of Albiac and Kalton
\cite{AlbiacKalton}. The main theorem shows that the difference between the
two linear structures is systematic and maximal: it occurs for every
infinite-dimensional $q$-Banach space $X$ with $p<q\leq1$.
In particular, the Banach-space examples include
\[
 \begin{split}
 \Fp(\ell_2)&\simeq_{\mathrm{Lip}}\ell_2\oplus_pZ_{\ell_2},\\
 \Fp(c_0)&\simeq_{\mathrm{Lip}}c_0\oplus_pZ_{c_0},\\
 \Fp(\ell_1)&\simeq_{\mathrm{Lip}}\ell_1\oplus_pZ_{\ell_1},\\
 \Fp(L_1[0,1])&\simeq_{\mathrm{Lip}}L_1[0,1]\oplus_pZ_{L_1[0,1]}.
 \end{split}
\]
In each line the spaces are Lipschitz isomorphic but not linearly
isomorphic. Taking $X$ to be an infinite-dimensional Banach space, we
deduce that none of the following properties is determined by the
Lipschitz structure of a $p$-Banach space:
\begin{enumerate}[label=\textup{(\roman*)}]
\item the Schur $p$-property;
\item $\ell_p$-saturation;
\item the absence of infinite-dimensional Banach, or more generally
$q$-Banach, subspaces for $p<q\leq1$;
\item the property that every operator from a Banach space into the given
space is compact;
\item the existence of a complemented copy of a prescribed
infinite-dimensional Banach space, for example $\ell_2$, $c_0$, $\ell_1$,
or $L_1[0,1]$.
\end{enumerate}
Indeed, $\Fp(X)$ has the first four properties and has no subspace as in
\textup{(v)}, while $X\oplus_pZ_X$ contains a complemented copy of $X$ and
therefore fails the corresponding assertions.

\subsection{Strong Schur \texorpdfstring{$p$}{p}-spaces with trivial dual}

\begin{theorem}\label{thm:trivial-dual}
Let $0<p<q<1$ and put
\[
 E_{p,q}=\Fp([0,1],d_q),\qquad d_q(s,t)=|s-t|^{1/q}.
\]
Then $E_{p,q}$ is an infinite-dimensional separable $p$-Banach space
with the following properties:
\begin{enumerate}[label=\textup{(\roman*)}]
\item it has the $K_{p,q}$-strong Schur $p$-property;
\item $E_{p,q}^*=\{0\}$ and its Banach envelope is zero;
\item it fails the approximation property;
\item it is $\ell_p$-saturated but has no complemented copy of $\ell_p$;
\item its $q$-Banach envelope is isometric to $L_q[0,1]$.
\end{enumerate}
\end{theorem}

\begin{proof}
Put $M=([0,1],d_q)$. This is a compact $q$-metric space, so
Theorem~\ref{thm:compact-strong} gives \textup{(i)}, and
Lemma~\ref{lem:Schur-saturated} gives $\ell_p$-saturation.
The Dirac vectors at rational points have dense linear span,
so $E_{p,q}$ is separable.

Since $M$ is already a $q$-metric space,
\cite[Proposition~4.20\textup{(a)}]{AACD} identifies the
$q$-Banach envelope of $E_{p,q}$ with $\mathcal F_q(M)$,
which is isometric to $L_q[0,1]$ by
\cite[Theorem~4.13]{AACD}. This proves \textup{(v)} and
also infinite dimension, since the envelope map has dense
range.

By \cite[Lemma~9.2 and Remark~9.5 in the arXiv version]
{AABWGreedy}, the dual of $E_{p,q}$ is isometric to
$L_q[0,1]^*$, and its Banach envelope is the Banach envelope
of $L_q[0,1]$. Both are zero because $q<1$, proving
\textup{(ii)}.

Every bounded linear map $T:E_{p,q}\to Y$ into a space
whose dual separates points is zero: indeed, $\phi T=0$
for every $\phi\in Y^*$.
Taking $Y$ to be the finite-dimensional range of a
finite-rank operator shows that every such operator on
$E_{p,q}$ is zero. Hence the identity cannot be approximated
even on a singleton containing a nonzero vector, proving
\textup{(iii)}.
Taking $Y$ to be a copy of $\ell_p$, whose coordinate
functionals separate points, excludes a bounded projection
onto $Y$. Together with $\ell_p$-saturation, this proves
\textup{(iv)}.
\end{proof}

Theorem~\ref{thm:trivial-dual} answers positively the question about
Schur $p$-spaces with trivial dual posed immediately after
\cite[Question~6.2]{AABCSchur}. It also shows that even the strong
Schur $p$-property implies neither the approximation property nor
the existence of a complemented copy of $\ell_p$.

\begin{corollary}\label{cor:compact-without-BAP}
There is a compact metric space $K$ such that $\Fp(K)$ has the strong
Schur $p$-property but fails the bounded approximation property.
\end{corollary}

\begin{proof}
By \cite[Corollary~2.2]{HLP}, there is a compact metric space $K$
whose Banach free space $\mathcal F(K)$ fails the approximation
property. Theorem~\ref{thm:compact-strong} gives the strong Schur
$p$-property of $\Fp(K)$.

We verify that the bounded approximation property passes to the Banach
envelope. Suppose that a quasi-Banach space $E$ has a net of finite-rank
operators $T_i:E\to E$ with $\sup_i\norm{T_i}\leq C$ and
$T_i x\to x$ for every $x\in E$. Let $J:E\to\widehat E$ be its
Banach-envelope map. By the universal property of the envelope, the
operators $JT_i:E\to\widehat E$ extend to operators
$\widehat T_i:\widehat E\to\widehat E$ with
$\norm{\widehat T_i}\leq C$. Their ranges lie in the
finite-dimensional spaces $J(T_i(E))$. Thus $\widehat T_i$ has finite rank.
Furthermore,
\[
 \widehat T_iJx=JT_ix\longrightarrow Jx\qquad(x\in E).
\]
The uniform bound and density imply convergence to the identity on
all of $\widehat E$. Hence $\widehat E$ has the bounded approximation property.

Applied to $E=\Fp(K)$, this would give the bounded approximation
property, and therefore the approximation property, of its envelope
$\mathcal F(K)$, a contradiction. Thus $\Fp(K)$ fails the bounded
approximation property.
\end{proof}

This last argument establishes failure of the bounded approximation
property for an ordinary compact metric base. Failure of the
approximation property itself is established in
Theorem~\ref{thm:trivial-dual} for a compact $q$-metric base with $q<1$.
\subsection{Answers to questions on the Schur \texorpdfstring{$p$}{p}-property}

The main theorem and the preceding applications settle several questions
from \cite[Section~6]{AABCSchur}. In the following corollary, $M$
is an ordinary metric space, as in those questions.

\begin{corollary}\label{cor:answers-open-questions}
The following assertions hold.
\begin{enumerate}[label=\textup{(\roman*)},leftmargin=*]
\item \cite[Question~6.1]{AABCSchur} has a positive answer: in particular,
$\Fp([0,1])$ has the Schur $p$-property.
\item \cite[Question~6.2]{AABCSchur} has a negative answer. Indeed,
$\Fp([0,1])$ has the Schur $p$-property, while its Banach envelope is
$\mathcal F([0,1])\simeq L_1[0,1]$, which does not have the Schur property.
\item \cite[Question~6.5]{AABCSchur} has a positive answer: every bounded
set in $\Fp(M)$ which fails Kalton's property in the sense of
\cite[Definition~5.2]{AABCSchur} contains a sequence equivalent
to the canonical basis of $\ell_p$.
\item \cite[Question~6.6]{AABCSchur} has a positive answer, even without
boundedness, completeness, or discreteness assumptions on $M$.
\item No infinite-dimensional Banach space embeds linearly into
$\Fp(c_0)$, answering the problem mentioned in the discussion following
\cite[Question~6.1]{AABCSchur}.
\item The question about a Schur $p$-space with trivial dual posed after
\cite[Question~6.2]{AABCSchur} has a positive answer, even with the strong
Schur $p$-property.
\end{enumerate}
\end{corollary}

\begin{proof}
Assertions \textup{(i)}, \textup{(iv)}, and \textup{(v)} follow directly
from Theorem~\ref{thm:main} and Corollary~\ref{cor:no-Banach}. For
\textup{(ii)}, the identification of the Banach envelope and the fact that
$\mathcal F([0,1])\simeq L_1[0,1]$ are recalled in the discussion preceding
\cite[Question~6.2]{AABCSchur}; the latter space does not have the Schur
property.

For \textup{(iii)}, let $W\subseteq\Fp(M)$ be bounded and suppose that it
does not have Kalton's property. By \cite[Lemma~5.3]{AABCSchur}, every
relatively compact subset of $\Fp(M)$ has Kalton's property. Hence $W$ is
not relatively compact and therefore contains an infinite uniformly
separated sequence. Theorem~\ref{thm:main} supplies an $\ell_p$-subsequence.

Finally, \textup{(vi)} follows from Theorem~\ref{thm:trivial-dual}.
\end{proof}

\section{Note on the current version}
\noindent The present version of the manuscript is being posted at a relatively early stage prior to its submission to a  journal. The mathematical arguments in the proofs are close to its final form,  but the manuscript is still being polished: in particular, the exposition will be streamlined and we intend to select and emphasize those results which we believe form the most interesting contribution. A more detailed Introduction with background on the problems we tackle will also be provided in a later version.

One reason for making the present version publicly available already at this stage is the rapidly increasing use of AI systems to attack explicitly stated open problems. Such tools can now produce mathematical arguments considerably faster than they can be carefully checked, organized, and developed into a coherent mathematical contribution. We therefore prefer to make our results and their present proofs publicly available while continuing to improve the manuscript, in particular so that the provenance and priority of the results are clear.

\section{Statements and declarations}

\subsection*{Conflict of Interest}
The authors declare that they have no conflict of interest.

\subsection*{Data Availability} 
Since no datasets were generated or analyzed during the current study, data sharing does not apply to this article.

\subsection*{Generative AI use}
During the development and preparation of this manuscript, the authors made use of AI-assisted tools as a conversational aid for exploring ideas and possible approaches, as well as for editorial assistance. The selection, formulation, and presentation of the mathematical results are those of the authors.

\end{document}